\documentclass[12pt,a4paper]{amsart}
\usepackage{amsfonts,amsmath,amssymb}
\usepackage{hyperref}
\usepackage[all]{xy}

\usepackage{amssymb,amsthm,amsxtra}
\usepackage[usenames,dvipsnames]{xcolor}

\usepackage{amscd}
\usepackage{amsthm}
\usepackage{amsfonts}
\usepackage{amssymb}
\usepackage{mathrsfs}
\newtheorem*{theorem*}{Theorem}
\newtheorem{lemma}{Lemma}[section]

\newtheorem*{proposition*}{Proposition}

\newtheorem{theorem}[lemma]{Theorem}

\newtheorem*{conjecture*}{Conjecture}
\newtheorem*{lemma*}{Lemma}

\newtheorem{thm}[lemma]{Theorem}

\newtheorem{lem}[lemma]{Lemma}
\newtheorem{cor}[lemma]{Corollary}

\newcommand{\alfa}{i}

\theoremstyle{remark}

\newtheorem*{remark*}{Remark}
\newtheorem*{remarks*}{Remarks}
\newtheorem{rem}[lemma]{Remark}
\newtheorem{example}[lemma]{Example}

\newtheorem*{examples*}{Examples}
\newtheorem*{example*}{Example}
\newtheorem{question}[lemma]{Question}

\theoremstyle{definition}

\numberwithin{equation}{section}

\baselineskip 18pt \textwidth 16cm  \theoremstyle{plain}

\newcommand{\Crit}{\operatorname{Crit}}

\newcommand{\Inc}{\operatorname{Inc}}

\newcommand{\A}{\mathbb{A}}

\newcommand{\Card}{\operatorname{Card}}

\newcommand{\R}{{\mathbb R}}
\newcommand{\bR}{{\mathbb R}}
\newcommand{\C}{{\mathbb C}}
\newcommand{\G}{{\mathbb G}}

\newcommand{\CW}{{\mathcal W}}

\newcommand{\cU}{\mathcal{U}}

\newcommand{\Rami}[1]{{{#1}}}

\newcommand{\RamiE}[1]{{{#1}}}

\newcommand{\RamiI}[1]{{{#1}}}

\newcommand{\RamiQ}[1]{{}}
\newcommand{\RamiQB}[1]{{}}

\newcommand{\oldFT}[1]{}

\newcommand{\oldWFHol}[1]{}

\usepackage{varioref}

\begin{document}

\email{drinfeld@math.uchicago.edu}

\thanks{Partially supported by NSF grant DMS-1063470.}

\dedicatory{To G\'erard Laumon on his 60th birthday}

\author{Vladimir Drinfeld}
\address{Vladimir Drinfeld, Department of Mathematics, University of Chicago.}

\title{Fourier transform of algebraic measures}

\maketitle
\begin{abstract}
These are notes of a talk based on the work \cite{AD} joint with A.~Aizenbud. 

Let $V$ be a finite-dimensional vector space over a local field $F$ of characteristic 0. 
Let $f$ be a function on $V$ of the form $x\mapsto \psi (P(x))$, where $P$ is a polynomial on $V$  and $\psi$ is a nontrivial additive character of $F$. Then it is clear that the Fourier transform Four$( f)$ is well-defined as a distribution on $V^*$. Due to J.Bernstein, Hrushovski-Kazhdan, and Cluckers-Loeser, it is known
that Four$( f)$ is smooth on a non-empty Zariski-open conic subset of $V^*$. The goal of these notes is to sketch a proof of this result (and some related ones), which is very simple modulo resolution of singularities (the existing proofs use D-module theory in the Archimedean case and model theory in the non-Archimedean one). 
\end{abstract}
\setcounter{tocdepth}{3}

These are notes of a talk based on the work \cite{AD} joint  with A.~Aizenbud.
The results from \cite{AD} are formulated in \S\S\ref{s:1}-\ref{s:3},
the proofs are sketched in   \S\S\ref{s:4}-\ref{s:sketch}. 

In Appendix \ref{s:non-Lagrangian} we discuss some ``baby examples"; this material is not contained 
in~\cite{AD}.

\bigskip

I thank D.~Kazhdan for drawing my attention to the questions considered in these notes.
I also thank M.~Kashiwara for communicating to me Example~\ref{ex:Kashiwara} from Appendix~\ref{s:non-Lagrangian}.

\section{A theorem on Fourier transform}  \label{s:1}

Let $F$ be a local field of characteristic 0 (Archimedean or not). 
Let $\psi :F\to\C^{\times}$ be a nontrivial additive character.

Let $V$ be a finite-dimensional vector space over $F$ and $P:V\to F$ a polynomial.
Then $\psi (P(x))$ is a smooth\footnote{In the non-Archimedean case ``smooth" means ``locally constant", in the non-Archimedean case the word ``smooth" is understood literally.} $\C$-valued function on $V$.

Consider $\widehat{\psi (P(x))}$, i.e., the Fourier transform of the function $\psi (P(x))$.

Note that in the naive sense the Fourier transform is not defined because $\psi (P(x))$ doesn't decay as 
$x\to\infty$, rather it oscillates. However,
\emph{$\widehat{\psi (P(x))}$ is well-defined as a distribution\footnote{Our conventions are as follows:
a distribution on a manifold $M$ is a generalized measure (i.e., a linear functional on the space of smooth 
\emph{functions} with compact support), while a generalized function on $M$  is a linear functional on the space of smooth \emph{measures} with compact support. If $M$ is a vector space then sometimes (but not here) we do not distinguish functions from measures.
} on $V^*$.}

In the non-Archimedean case this is clear because the Fourier transform in the sense of distributions is well-defined for \emph{any} generalized function, in particular, for any smooth function.

In the Archimedean case the Fourier transform is well-defined for generalized functions of \emph{moderate growth},
and of course,  $\psi (P(x))$ has moderate growth. 

\begin{theorem}   \label{t:1}
There exists a Zariski-open $U\subset V^*$, $U\ne\emptyset$, such that the distribution $\widehat{\psi (P(x))}$ is smooth on~$U$.
\end{theorem}

If $F$ is Archimedean Theorem~\ref{t:1} was proved by J.~Bernstein \cite{Ber_FT} using \emph{D-module theory}. Moreover, he proved that $\widehat{\psi (P(x))}$ satisfies a holonomic system of linear p.d.e.'s with polynomial coefficients. 

If $F$ is non-Archimedean Theorem~\ref{t:1} was proved by Kazhdan-Hrushovski \cite{HK} and Cluckers-Loeser  \cite{CL} using \emph{model theory}.
In both articles Theorem~\ref{t:1} appears as one of many corollaries of a general theory, and this general theory is quite different from D-module theory used by Bernstein.

The goal of these notes is to explain another proof of Theorem~\ref{t:1} and its refinements, namely the one from \cite{AD}. 
It works equally well in the Archimedean and non-Archimedean case. Unlike the older proofs, it uses resolution of singularities.\footnote{In my talk I said that a variant of 
``local uniformization" (see \cite{Za}, \cite{ILO}) would suffice. But the argument that I had in mind contained a gap.}
On the other hand, once you believe in resolution of singularities, the rest is an exercise in elementary analysis (with a  bit of elementary symplectic geometry). 

\medskip

Here are some refinements of Theorem \ref{t:1}.

\medskip

 Refinement A.\, The open subset $U$ can be chosen to be independent of $\psi$.

Refinement B.\,  The open subset $U$ can be chosen to be defined over the field $K$ generated by the coefficients of $P$.

Refinement B$'$.\, The open subset $U$ can be chosen to work for all embeddings of $K$ into all possible local fields. (This makes sense by virtue of A.)

\section{A theorem that implies Theorem~\ref{t:1}.}   \label{s:2}

Let $W$ be a finite-dimensional vector space over $F$.
Let $X$ be a smooth algebraic variety over $F$ and $\varphi :X\to W$ a proper morphism.
Let $\omega$ be a regular top differential form on $X$.
Then we have a measure 
$|\omega|$ on $X(F)$. Set $\mu:=\varphi_*|\omega|$ (note that $\varphi_*$ is well-defined because
$\varphi$ is proper); $\mu$ is a measure on the vector space $W$, in particular, it is a distribution.
Its Fourier transform, $\hat\mu$, is a well-defined\footnote{In the Archimedean case one has to check that $\mu$ has moderate growth. This is not hard and well known.} generalized function on $W^*$; it depends on the choice of $\psi$. The next theorem is
an analog of Theorem~\ref{t:1} and its  Refinements A,B,B$\,'$. 

\begin{thm}   \label{t:2}
There exists a non-empty Zariski-open $U\subset W^*$, independent of $\psi$, such that $\hat\mu$ is smooth on~$U$.
Moreover, if $(X , \varphi ,\omega )$ is defined over a subfield 
$K\subset F$ then one can choose $U$ to be defined over $K$ and to have the required property for all embeddings of $K$ into all possible local fields.
\end{thm}

\begin{rem}   \label{r:homoth}
In Theorem~\ref{t:2} independence of $U$ on $\psi$ is equivalent to \emph{stability of $U$ under homotheties} of $W$. (This was not the case in the situation of Theorem~\ref{t:1} because $\psi$ occurred there twice: in the definition of Fourier transform and in the expression $\psi (P(x))$.)
\end{rem}

Let us show that Theorem~\ref{t:2} implies Theorem~\ref{t:1} and its refinements formulated at the end of 
\S\ref{s:1}. To prove this,  apply Theorem~\ref{t:2} as follows.
Set $W:=V\oplus F=V\times F$,  $X:=V$. Define
$\varphi :V\to V\times F$ by $\varphi(v):= (v, P(v))$.
Take $\omega$ to be an invariant differential form on $V=X$. 

Then the generalized function $\hat\mu$
on $W^*=V^*\times F$ is equal to the continuous map $$F\to \{\mbox{generalized functions on } V^*\}$$
that takes $\eta\in F$ to the Fourier transform of $\psi (\eta\cdot P(x))$ with respect to $x\in V$.

So Theorem~\ref{t:2} says that the Fourier transform of $\psi (\eta\cdot P(x))$ with respect to \emph{both} 
$x$ and $\eta$ (which is a priori a \emph{generalized} function on $V^*\times F$)
 is, in fact, \emph{smooth} on some non-empty open subset
 $U\subset V^*\times F$, $U\ne\emptyset$, which can be chosen to be stable with respect to homotheties
 (see Remark~\ref{r:homoth}). It remains to note that such $U$
 has non-empty intersection with the hyperplane $V^*\times \{1\}\subset V^*\times F$.
 
 \medskip
 
Theorem~\ref{t:2} is deduced from Theorem~\ref{t:3}, which is formulated in the next section.

\section{The theorem on wave front sets}   \label{s:3}
\subsection{Isotropic subsets of symplectic varieties}   \label{ss:isotropic_notion}
Let $M$ be a symplectic algebraic mani\-fold\footnote{``Manifold"=``smooth variety".} over a field of characteristic 0. An algebraic subvariety (or more generally, a constructible subset) of $M$ is said to be \emph{isotropic} if each of its smooth subvarieties is. (As usual, a smooth subvariety of $M$ is said to be isotropic if each of its tangent spaces is). It is known that the closure of an isotropic subvariety is isotropic (e.g., see \cite[Proposition 1.3.30]{CG} and  \cite[\S 1.5.16]{CG}). So a subvariety of a symplectic variety is isotropic if and only if its smooth locus is.


Now suppose that $M$ is the cotangent bundle of an  algebraic manifold $Y$. A subset of $M=T^*Y$ is said to be \emph{conic} if it is stable with respect to the action of $\G_m$ on $T^*Y$. The following two statements are easy and standard.

\begin{lem}    \label{l:conic isotropic}
A conic algebraic subvariety $Z\subset T^*Y$ is isotropic if and only if there exists a finite collection of smooth locally closed subvarieties of $Y$ such that $Z$ is contained in the union of their conormal bundles.
\end{lem}

\begin{cor}    \label{c:conic isotropic}
Let $Z\subset T^*Y$ be an isotropic conic algebraic subvariety. Then there exists a dense open subset 
$U\subset T^*Y$ such that $Z\cap T^*U$ is contained in the zero section.
\end{cor}

\subsection{The theorem on wave fronts}  
Let $F$ be a local field of characteristic 0. Let $Y$ be an analytic manifold\footnote{The words ``analytic manifold" are understood in the most naive sense (not rigid-analytic or Berkovich-analytic).} over $F$. 
(If $F=\R$ one can assume $Y$ to be a $C^{\infty}$ manifold.) Following L.~H\"ormander \cite{Hor} and
D.~Heifetz \cite{Hef}, to each distribution or generalized function $\nu$ on $Y$ one associates a conic closed subset of $T^*Y$ called the \emph{wave front set} of~$\nu$. The precise definition will be recalled in 
\S\ref{ss: wave front def} below.


In the situation of Theorem~\ref{t:2} $\hat\mu$ is  a generalized function on $W^*$, so its wave front is a subset of  $T^*W^*$. Note that $T^*W^*=W^*\times W$ carries an action of $\G_m\times\G_m\,$.

\begin{theorem}   \label{t:3}
In the situation of Theorem~\ref{t:2} the wave front of $\hat\mu$ is contained in $I(F)$, where  $I\subset T^*W^*$ is some isotropic algebraic subvariety stable with respect to $\G_m\times\G_m$ and defined over $K$ (as before, $K$ is a field of definition of our data).
Moreover, one can choose $I$ to have the required property for all embeddings of $K$ into all local fields.
\end{theorem}

Note that Theorem~\ref{t:3} implies Theorem~\ref{t:2}. This is a consequence of Corollary~\ref{c:conic isotropic} and the following key property of the wave front of a generalized function $\nu$ on a manifold~$Y$:  the restriction of $\nu$ to an open subset $U\subset Y$ is smooth if and only if the intersection of the wave front with $T^*U$ is contained in the zero section.

The next subsection can be skipped by the reader.

\subsection{Definition of wave front set}   \label{ss: wave front def}
The definition that we are using slightly differs from the classical definition of L.~H\"ormander \cite{Hor}.

Let $V$ be a finite-dimensional $F$-vector space and let $\nu$ be a distribution on an open subset 
$U\subset V$. The \emph{wave front} of $\nu$ is a certain closed subset in $T^*U=U\times V^*$.
Namely, a point $(x_0, w_0)\in U\times V^*$ is \emph{not} in the wave front of $\nu$ if
there exists a smooth compactly supported function $\rho$ on $V$ with $\rho(x_0)\neq 0$ such that 
the Fourier transform\footnote{The Fourier transform is a smooth function on $V^*$.}  of $\rho \cdot \nu$ ``vanishes asymptotically" in the direction of $w_0$. 

The precise meaning of these words is as follows: we  say that a function $f$ on $V^*$  \emph{vanishes asymptotically in the direction of} $w_0$
if there exists a smooth compactly supported function $\sigma$ on $V^*$ with $\sigma (w_0) \neq 0$ such that the function $\phi$ on $V^* \times F$ defined by $\phi(w,\lambda):=f(\lambda w) \cdot \sigma (w)$ is a Schwartz function.\footnote{Thus we understand the above word ``direction" as a line. H\"ormander considers $F=\R$ and understands ``direction" as a half-line (i.e., a ray).}

In other words, if $w_0\ne 0$ and $F$ is non-Archimedean we require the Fourier transform of 
$\rho \cdot \nu$  to become compactly supported
after restricting to a small open  $F^{\times}$-stable neighborhood of $w_0$. In the Archimedean case compact support is replaced by rapid decay. If $w_0=0$ we require that $\rho \cdot \nu =0$. 

According to \cite{Hor} and \cite{Hef}, the above notion of wave front set is invariant with respect to changes of variables, so it makes sense for distributions (or generalized functions) on manifolds.

For more details, see  \cite[Appendix A]{AD}, \cite{Hor}, and \cite{Hef}.

%

\subsection{Warning}   \label{ss:warning}
It is well known that the characteristic variety of a coherent $\mathcal D$-module is always coisotropic
(in particular, if it is isotropic then it is Lagrangian). On the other hand, a similar statement for wave front sets is \emph{false}, see Lemma~\ref{l:not Lagrangian} and Example~\ref{ex:F} from 
Appendix~\ref{s:non-Lagrangian}.

\section{Construction of $U\subset W^*$ and $I\subset T^*W^*$.}   \label{s:4}
In this section we construct the sets $U$ and $I$ whose existence is claimed in Theorems~\ref{t:2} 
and \ref{t:3}.

\medskip

Let $\overline{W}$ denote the space of  lines in $W\oplus F$; in other words, $\overline{W}$ is the projective space containing $W$ as an open subspace. Set $W_{\infty}:=\overline{W}\setminus W$; this is the hyperplane at infinity.

By assumption, the differential form $\omega$ has no poles on $X$, but it may have zeros. 
Set $X^{\circ}:=\{ x\in X\,|\,\omega (x)\ne 0\}$. In what follows we assume that $X^{\circ}$ is dense in $X$
(so $X\setminus X^{\circ}$ is a divisor).

By Hironaka's theorem, after replacing the variety $X$ by a certain modification\footnote{A modification of $X$ is a variety equipped with a proper morphism to $X$ which is a birational isomophism.} of it, $\varphi :X\to W$ admits a compactification to $\overline{\varphi} :\overline{X}\to \overline{W}$ with $\overline{X}$ \emph{smooth} and 
$\overline{X}\setminus X^{\circ}$ being a NC divisor (or even an SNC divisor). As usual, ``NC" (resp.``SNC") stands for ``normal crossings" (resp. ``strict normal crossings"); the word ``strict" means here that each irreducible component of the divisor is smooth. 

So one has a commutative diagram 
\begin{equation}
\xymatrix
{        X   \ar@{}[r]|-*[@]{\hookrightarrow}   \ar@{->}_{\varphi}[d]        & \overline X\ \ar@{->}_{\overline \varphi}[d]
\ar@{}[r]|-*[@]{\supset}   & X_{\infty}:=\overline \varphi^{-1}(W_{\infty})\ \ar@{->}_{}[d] \\
          W \    \ar@{}[r]|-*[@]{\hookrightarrow}   &  \overline W\ar@{}[r]|-*[@]{\supset} & W_{\infty} \\ }
\end{equation}

\bigskip

Set $D:=\overline{X}\setminus X^{\circ}$. By assumption, $D$ is a NC divisor. Note that
$X_{\infty}$ is  a divisor contained in $\overline{X}\setminus X^{\circ}$, so $X_{\infty}$ is also NC.


We choose the modification of the original $X$ and its compactification to be defined over
the small field~$K$. 

The open subset $U\subset W^*$ from Theorem~\ref{t:2} and the isotropic subvariety $I$ from 
Theorem~\ref{t:3} are constructed very explicitly in terms of the above choices, and it will be clear that $U$ and $I$
are defined over $K$. To construct $U$ and $I$, we need some notation.


Let $\nu:\hat D\to D$ be the normalization. For each $r\ge 1$ let $Z'_r$ be the normalization of
\[
\{ x\in D\, |\, \Card \nu^{-1}(x)\ge r \}.
\]
In particular, $Z'_1=\hat D$ and $Z'_r=\emptyset$ if $r>\dim X$. It is easy to see that $Z'_r$ is smooth.
If $D$ is an SNC divisor with irreducible components $D_j$, $j\in J$, then  $Z'_r$ is just the disjoint union of the intersections $D_S:=\bigcap\limits_{j\in S} D_j\,$ corresponding to all subsets $S\subset J$ of order $r$.

Let $Z_r\subset Z'_r$ denote the disjoint union of those connected components of $Z_r$ whose image in 
$\overline X$ is contained in $X_\infty\,$. The map $\overline{\varphi}:\overline{X}\to\overline{W}$ induces a map $\varphi_r:Z_r\to W_{\infty}\,$. 


\medskip

\noindent\emph{Definition of $U$.} Recall that the projective space $W_\infty$ is the space of 1-dimensional subspaces in $W$.  The dual projective space $W_{\infty}^*$ equals $(W^*\setminus\{ 0\})/\G_m\,$; on the other hand, points of $W_{\infty}^*$ can be considered as projective hyperplanes $H\subset W_{\infty}\,$.
Let $\tilde U\subset W_{\infty}^*$ denote the set of those projective hyperplanes $H\subset W_{\infty}$ that are
transversal\footnote{Transversality means that $Z_r\times_{W_{\infty}}H$ is a smooth divisor in $Z_r\,$.}
 to $\varphi_r:Z_r\to W_{\infty}$ for each $r$. Finally, define 
$U\subset W^*\setminus\{ 0\}$ to be the preimage of $\tilde U$.




\medskip

To describe the isotropic subset $I$, let us recall a general construction: to a morphism 
$f:X\to Y$ between smooth algebraic varieties one associates a conic subset $\Crit_f\subset T^*Y$, namely
\begin{equation}   \label{e:defCrit}
\Crit_f\:=\{ (y,\xi )\in T^*Y\,|\,\exists x\in f^{-1} (y) \mbox{ such that } (d_xf)^*(\xi )=0\}.
\end{equation}
Here we understand the quantifier $\exists x$ \emph{in the sense of algebraic geometry} (i.e., the field of definition of $x$ may be \emph{larger} than the ground field).

It is well known that $\Crit_f$ is isotropic\footnote{On the other hand, $\Crit_f$ is not necessarily Lagrangian, see Example \ref{ex:Kashiwara} from Appendix~\ref{s:non-Lagrangian}.} (e.g., this follows from\cite[Prop.~ 2.7.51]{CG} or \cite[Lemma~1]{G}). If $f$ is proper then $\Crit_f$ is closed.

\medskip

\noindent\emph{Definition of $I$.} 
Applying the above construction to $\varphi_r:Z_r\to W_{\infty}$ we get a conic isotropic subvariety
$\Crit_{\varphi_r}\subset T^*W_{\infty}$. Set $I':=\bigcup\limits_r\Crit_{\varphi_r}\subset T^*W_{\infty}$.
Note that $W_{\infty}=(W\setminus\{ 0\})/\G_m$, so $T^*W_{\infty}$ is a 
subquotient of $T^*(W\setminus\{ 0\})=(W\setminus\{ 0\})\times W^*$. So $I'$ defines a subset 
$\tilde I\subset (W\setminus\{ 0\})\times W^*$. Now set 
$I=\tilde I\bigcup (W\times\{ 0\})\bigcup (\{ 0\}\times W^*).$


\begin{theorem}   \label{t:4}
The sets $U$ and $I$ defined above have the properties required in Theorems~\ref{t:2} and \ref{t:3}: namely,
the generalized function $\hat\mu$ has smooth restriction to $U$ and the wave front of $\hat\mu$ is contained in $I$. 
\end{theorem}

The proof will be sketched in the next section. A complete proof (under the assumption that $\overline{X}\setminus X^{\circ}$ is a divisor with \emph{strict} normal crossings) is given in \cite{AD}.

Note that the claim about $U$ from Theorem~\ref{t:4} follows from the claim about $I$. This is clear from the next remark.


\begin{rem}
The subset $I'\subset T^*W_{\infty}$ from the definition of $I$ and the subset $\tilde U\subset W_{\infty}^*$ from
the definition of $U$ are related as follows.
Identifying $(T^*W_{\infty}-0)/\G_m$ with the incidence relation $\Inc\subset W_{\infty}\times W_{\infty}^*$ we get from $I'$ a subset of $\Inc$. Its image in $W_{\infty}^*$ equals $W_{\infty}^*\setminus \tilde U$.
%
\end{rem}

\section{Sketch of the proof of Theorem~\ref{t:4}}   \label{s:sketch}
\subsection{Strategy}    \label{ss:strategy}

We are studying $\hat\mu$, where $\mu:=\varphi_*|\omega|$. Let $\overline{X}$ be as in \S\ref{s:4}. 
Decompose $\varphi :X\to W$ as
\begin{equation}     \label{e:2}
\xymatrix
{
X\ar[dr]_{\varphi} \ar@{^{(}->}[r]^{\alfa} &\overline{X}\times W\ar[d]^{\pi}\\
&W
}
\end{equation}
Note that since $\varphi$ is proper the map $i$ is a \emph{closed} embedding. Thus all three maps in diagram \eqref{e:2} are proper.

We have $\mu:=\varphi_*|\omega |=\pi_*i_*|\omega |$. So $\hat\mu=\pi_*\widehat{\alfa_*|\omega|}$. Here 
$\alfa_*|\omega|$ lives on
$\overline{X}\times W$,  and $\widehat{\alfa_*|\omega|}$ is  the \emph{partial Fourier transform} of 
$\alfa_*|\omega|$ with respect to  $W$.

The strategy is to first find an upper bound for the wave front of $\widehat{\alfa_*|\omega|}$
and then to get from it an upper bound for the wave front of $\hat\mu$. The latter is straightforward.\footnote{The calculus of wave fronts was developed by H\"ormander precisely to make such computations straightforward.}
In \S\S\ref{ss:calculus}-\ref{ss:toric} below we discuss the problem of finding an upper bound for the wave front of $\widehat{\alfa_*|\omega|}$; this problem is, of course, local with respect to $\overline{X}$.


\subsection{A calculus-style formula for $\widehat{\alfa_*|\omega|}$}  \label{ss:calculus}
Let us get rid of the notation $i_*$ (which is not a part of a standard calculus course).

Let $x_0\in \overline{X}$. 
Our $\varphi :X\to W$ extends to $\overline{\varphi} : \overline{X}\to \overline{W}$, and on some open subset
$\cU\subset \overline{X}$ containing $x_0$ one
can write $\overline{\varphi} :\cU\to \overline{W}$   as 
$$\overline{\varphi}  (x)=(f(x):p(x)), $$
where $f$ is a $W$-valued regular function, $p$ is a scalar regular function, and the two functions have no common zeros. We claim that on $\cU\times W^*$ one has
\begin{equation}   \label{e:3}
\widehat{\alfa_*|\omega|}=\psi (\frac{\langle f(x) ,\xi \rangle}{p(x)})\cdot |\omega |, \quad (x,\xi )\in
\cU\times W^*.
\end{equation}
Here and in similar situations below we allow a slight abuse of notation: we write simply $\omega $ to denote the pullback of $\omega$ from $\overline{X}$ to $\cU\times W^*$, and we tacitly assume restriction from
$\overline{X}\times W^*$ to $\cU\times W^*$ in the l.h.s. of \eqref{e:3}.


Let us explain the precise meaning of formula \eqref{e:3}. To simplify the discussion, let us fix a Haar measure\footnote{If you do not fix a Haar measure on $W^*$ then the l.h.s. of ~\eqref{e:3} is a distribution along 
$\cU$ and a generalized function along $W^*$.} on $W^*$, then the l.h.s. of ~\eqref{e:3} is a distribution on 
$\cU\times W^*$.


The r.h.s. of ~\eqref{e:3} clearly makes sense as a distribution on a smaller set $(\cU\cap X)\times W^*$ (since $p$ has no zeros on $\cU\cap X$ we just have a  smooth function times $|\omega |$). But as explained below, the r.h.s. of ~\eqref{e:3} makes sense as a distribution on the whole set $\cU\times W^*$, and 
formula  \eqref{e:3} is, in fact, an equality of distributions on $\cU\times W^*$.

Here are two equivalent ways to define the r.h.s. of ~\eqref{e:3} as a distribution on $\cU\times W^*$.
First way: for a smooth compactly supported test function $h$ on $\cU\times W^*$, interpret 
\[
\int\limits_{x,\xi} h(x,\xi )\cdot \psi (\frac{\langle f(x) ,\xi \rangle}{p(x)})\cdot |\omega |
\]
as $\int\limits_{x}\int\limits_{\xi}$ (i.e., integrate along $\xi$ first\footnote{After integrating along $\xi$ one gets a smooth compactly supported function on $\cU$. Moreover, this function vanishes on the locus $p(x)=0$. In the Archimedean case all its derivatives vanish there as well.}).  

To explain the second way, let us assume, for simplicity, that $\cU$ is so small that there is a regular top form 
$\omega_0$ on $\cU$ without zeros. The problem is then to define the expression
\begin{equation}   \label{e:to_define}
\psi (\frac{\langle f(x) ,\xi \rangle}{p(x)})\cdot |\frac{\omega}{\omega_0} |
\end{equation}
as a generalized function on $\cU\times W^*$. In fact, one defines \eqref{e:to_define} as a map
\begin{equation}   \label{e:the_map}
\cU\to\{\mbox{generalized functions on }W^* \}.
\end{equation}
Namely, if $p(x)\ne 0$ then \eqref{e:to_define} is understood literally, and if $p(x)= 0$ then the corresponding
generalized function on $W^*$ is defined to be zero. One checks that the above map \eqref{e:the_map} is \emph{continuous}, so it defines a generalized function on $\cU\times W^*$.

\subsection{Using formula  \eqref{e:3}}  \label{ss:reducing}
The formula reduces the problem to finding an upper bound for the wave front of the distribution
\begin{equation}   \label{e:three}
\psi (\frac{\langle f(x) ,\xi \rangle}{p(x)})\cdot |\omega |, \quad (x,\xi )\in
\cU\times W^*.
\end{equation}

Recall that $f$ and $p$ have no common zeros. Note that the wave front of \eqref{e:three} does not change after replacing $f(x)$ by $f(x)+p(x)w_0$, where $w_0\in W$ is fixed. These facts imply that without loss of generality \emph{we can assume that $f$ has no zeros on $\cU$.} 
%

Now consider the expression
\begin{equation}   \label{e:4}
\psi (\frac{\eta}{p(x)})\cdot |\omega |, \quad (x,\eta )\in\cU\times F,
\end{equation}
which is simpler than \eqref{e:three}. Define \eqref{e:4} as a distribution on the whole $\cU\times F$ (rather than on the open subset where $p(x)\ne 0$) using the procedure from \S\ref{ss:calculus}.
%
The distribution \eqref{e:three} is the pullback of \eqref{e:4} with respect to the map 
$\cU\times W^*\to\cU\times F$ defined by $\eta=\langle f(x) ,\xi \rangle$ (the pullback is well-defined because
$f$ has no zeros on $\cU$, which implies that the map $\cU\times W^*\to\cU\times F$ is a submersion).
Thus it remains to find an upper bound for the wave front of \eqref{e:4}.

\subsection{Using toric symmetry}  \label{ss:toric}
The good news is that $f$ does not appear in formula \eqref{e:4}. Because of this, the toric symmetry due to the normal crossings assumption becomes manifest.
Let us explain more details.

\subsubsection{The SNC case}    \label{ss:SNC case} 
Let us first assume that the divisor $\overline{X}\setminus X^{\circ}$ has 
\emph{strict} normal crossings (this case is enough to prove Theorem~\ref{t:3}).
In this case 
we can pretend that $\cU=\A^n$ and that 
$p$ and $\omega$ from formula~\eqref{e:4} are given by \emph{monomials}. Then the problem is to give an upper bound for the wave front of the generalized function\footnote{Again, to define \eqref{e:8} as a generalized function on the whole space $F^n\times F$ (rather than outside of the coordinate hyperplanes) we use the procedure from 
\S\ref{ss:calculus}.}
\begin{equation}   \label{e:8}
u(x,\eta)=\psi (\frac{\eta}{x^{\alpha}})\cdot |x^{\beta} |, \mbox{ where } x=(x_1,\ldots,x_n)\in F^n, \eta\in F.
\end{equation}
As usual, $\alpha$ and $\beta$ are multi-indices and $x^{\alpha}:=x_1^{\alpha_1}x_2^{\alpha_2}
\ldots x_n^{\alpha_n}$. We assume that $\alpha_i\ge 0$ and whenever $\alpha_i= 0$ we have $\beta_i\ge 0$; the latter assumption ensures that \eqref{e:8} is well-defined as a generalized function.

\medskip

\noindent {\bf Key Lemma.} 
The wave front of \eqref{e:8} is contained in the union of the zero section of the cotangent bundle of $F^{n+1}$ and the conormal bundles of the coordinate planes of dimensions 
$0,1,\ldots, n$.

\smallskip

\begin{proof}
Let $\CW$ denote the wave front of $u$. 
The generalized function $u$ is quasi-invariant with respect to the action of the torus $T:=(F^{\times})^n$ on $F^{n+1}$ defined by
\[
\tilde x_i:=\lambda_i\cdot x_i,\quad \tilde\eta:=\lambda^{\alpha}\cdot\eta, \quad\quad 
(\lambda_1,\ldots,\lambda_n)\in(F^{\times})^n.
\]
The quasi-invariance property of $u$ implies the following property of $\CW$: suppose that $(z,\alpha)\in\CW$, where $z\in\A^{n+1}$ and $\alpha\in T^*_z(\A^{n+1})$; then $\alpha$ vanishes on the tangent space to the $T$-orbit containing $z$. Combining this property of $\CW$ with the fact that $u$ is smooth
outside of the union of the coordinate hyperplanes, we get the desired statement. 
\end{proof}

\subsubsection{The general case}

Now let us drop the strictness assumption on the divisor $\overline{X}\setminus X^{\circ}$.
Then instead of \eqref{e:8} one has to consider the following 
generalized function on $$E_1\times\ldots E_n\times F,$$ 
where $E_1,\ldots,E_n$ are finite extensions of $F$:
\[
u(x_1,\ldots ,x_n,\eta)=\psi (\frac{\eta}{N_{E_1/F}(x_1)^{\alpha_1}\cdot\ldots\cdot N_{E_n/F}(x_n)^{\alpha_n}})\cdot 
\prod_i |N_{E_i/F}(x_i)|^{\beta_i}, 
\;\quad x_i\in E_i\, ,\,  \eta\in F.
\]
An analog of the Key Lemma from \S\ref{ss:SNC case} for this generalized function still holds.
The proof from \S\ref{ss:SNC case} remains valid with an obvious change: instead of the torus 
$(F^{\times})^n$ one uses the torus $E_1^{\times}\times\ldots \times E_n^{\times}$.


\appendix

\section{How non-Lagrangian isotropic varieties appear} \label{s:non-Lagrangian}

\subsection{The isotropic subset $\Crit_f$ is not necessarily Lagrangian}
To a morphism $f:X\to Y$ between smooth algebraic varieties one associates an isotropic conic subset 
$\Crit_f\subset T^*Y$, see formula~\eqref{e:defCrit}. The subset $\Crit_f\subset T^*Y$ is \emph{not necessarily Lagrangian,} i.e., its closure may have components of dimension less than $\dim Y$. I learned the following example of this phenomenon from M.~Kashiwara.

\begin{example}    \label{ex:Kashiwara}
Define $f:\A^2\to \A^2$ by
$$f(t,x)=(t,t^nx), \quad n\ge 2.$$
The differential of $f$ never vanishes; it is degenerate if and only if $t=0$. So far we have used that $n\ge 1$.
Now using that $n\ge 2$ we see that
$\Crit_f$ is the union of the zero section and a line in the 
cotangent space of $(0,0)\in Y$. 
\end{example}

\begin{question}
Is there a \emph{proper} morphism $f:X\to Y$ between smooth algebraic varieties
such that $\Crit_f$ is not Lagrangian?
\end{question}

\subsection{A distribution whose wave front is isotropic but not Lagrangian. }
Consider the map $f:\A^2\to \A^2$ from Example~\ref{ex:Kashiwara}, i.e.,   
\begin{equation}   \label{e:f}
f(t,x)=(t,y), \mbox{ where } y=t^nx, \quad n\ge 2.
\end{equation}
Let $F$ be a local field of characteristic 0 and $\varphi :F\to\C$ 
a smooth compactly supported function. 
Define a distribution $u$ on $F^2$
by
\begin{equation}   \label{e:u}
u:=f_*(\varphi (x)\cdot |dt\wedge dx|)
\end{equation}
(the r.h.s makes sense because the restriction of $f:F^2\to F^2$ to the support of the distribution 
$\varphi (x)\cdot |dt\wedge dx|$ is proper). 
As usual, the distribution $u$ can be considered as a generalized function. This function is smooth
outside of the point $y=t=0$; in fact, it is easy to see that 
$$u(t,y)=|t^{-n}|\cdot\varphi (y/t^n) \;\;\mbox{ if } t\ne 0; \quad u(t,y)=0 \;\;\mbox{ if } t=0, y\ne 0.$$

\begin{lem} \label{l:not Lagrangian}
If $\varphi\ne 0$ then the wave front of $u$ is the union of the following two sets:

a) the part of the zero section that corresponds to the support of the function $\varphi (y/t^n)$;

b) a line\footnote{You see this line (at lest, in the case $F=\bR$) if you draw a picture of the function $|t^{-n}|\cdot\varphi (y/t^n)$ (more precisely, a picture of the set $|y|\le c\cdot |t^n|$, where $c$ is a ``fuzzy" number rather than a ``sharp" one).} in the cotangent space of the  point $y=t=0$.
\end{lem}

\begin{proof}
The calculus of wave fronts\footnote{E.g., see \cite[Prop.~2.3.8]{AD} and  \cite[Def.~2.3.6]{AD}.} tells us that the wave front of $u$ is contained in the set of $F$-points of $\Crit_f\,$. So the description of $\Crit_f$ given in Example~\ref{ex:Kashiwara} yields an upper bound for the wave front. On the other hand, the wave front set is a conic subset of the cotangent bundle which is not contained in the zero section (because $u$ is not smooth at the point $y=t=0$).
\end{proof}

\begin{example}    \label{ex:F}
Here is a variant of the above construction assuming that $F\ne\C$.  Choose 
$a\in F^\times$, $a\not\in (F^\times )^2$. 
Instead of the map \eqref{e:f}, consider
the map $f:F^2\to F^2$ defined by
\begin{equation} \label{e:new f}
f(t,x)=(t,\frac{x^3}{3}-at^2x).
\end{equation}
Consider the distribution \eqref{e:u}  on $F^2$ corresponding to the new map $f:F^2\to F^2$. Then Lemma~\ref{l:not Lagrangian} holds for this distribution (to see this, note that
since $a\not\in (F^\times )^2$ the only critical point of $f:F^2\to F^2$ is $x=t=0$). Note that the algebraic variety $\Crit_f$ corresponding to the map \eqref{e:new f} is Lagrangian, but one of its irreducible components has very few $F$-points.
\end{example}


\end{document}